\documentclass{amsart}
\usepackage{amsmath, amsthm, amsfonts, amsaddr, enumerate, fullpage}
\usepackage[all]{xy}
\usepackage[inline]{enumitem}
\newtheorem{thm}{Theorem}[section]
\newtheorem{lem}[thm]{Lemma}
\newtheorem{cor}[thm]{Corollary}
\newtheorem{con}[thm]{Conjecture}
\newtheorem{prop}[thm]{Proposition}

\theoremstyle{definition}

\newtheorem{definition}[thm]{Definition}

\theoremstyle{remark}
\newtheorem{case}{Case}

\numberwithin{subcase}{case}

\begin{document}
\title{List rankings and on-line list rankings of graphs}
\author{Daniel C. McDonald}
\address{Department of Mathematics, University of Illinois, 
Urbana, IL, USA}
\email{dmcdona4@illinois.edu}
\date{}
\maketitle 

\begin{abstract}
A $k$-ranking of a graph $G$ is a labeling of its vertices from $\{1,\ldots,k\}$ such that any nontrivial path whose endpoints have the same label contains a larger label.  The least $k$ for which $G$ has a $k$-ranking is the ranking number of $G$, also known as tree-depth.  The list ranking number of $G$ is the least $k$ such that if each vertex of $G$ is assigned a set of $k$ potential labels, then $G$ can be ranked by labeling each vertex with a label from its assigned list.  Rankings model a certain parallel processing problem in manufacturing, while the list ranking version adds scheduling constraints.  We compute the list ranking number of paths, cycles, and trees with many more leaves than internal vertices.  Some of these results follow from stronger theorems we prove about on-line versions of list ranking, where each vertex starts with an empty list having some fixed capacity, and potential labels are presented one by one, at which time they are added to the lists of certain vertices; the decision of which of these vertices are actually to be ranked with that label must be made immediately.
\end{abstract}
\section{Introduction}
\label{sec:intro}
We consider a special type of proper vertex coloring using positive integers, called ``ranking.''  As with proper colorings, there exist variations on the original ranking problem.  In this paper we consider the list ranking problem, first posed by Jamison in 2003 \cite{J}.
\begin{definition}
Let $G$ be a finite simple graph, and let $f:V(G)\rightarrow\mathbb{N}$.  An \emph{$f$-ranking} $\alpha$ of $G$ labels each $v\in V(G)$ with an element of $\{1,\ldots,f(v)\}$ in such a way that if $u\neq v$ but $\alpha(u)=\alpha(v)$, then every $u,v$-path contains a vertex $w$ satisfying $\alpha(w)>\alpha(u)$ (equivalently, every path contains a unique vertex with largest label).  For a positive integer $k$, a \emph{$k$-ranking} of $G$ is an $f$-ranking for the constant function $f=k$.  The \emph{ranking number} of $G$, denoted here by $\rho(G)$ (though in the literature often as $\chi _r(G)$), is the minimum $k$ such that $G$ has a $k$-ranking. 
\end{definition}
Vertex rankings of graphs were introduced in \cite{IRV}, and results through 2003 are surveyed in \cite{J}.  Their study was motivated by applications to VLSI layout, cellular networks, Cholesky factorization, parallel processing, and computational geometry.  Vertex rankings are sometimes called ordered colorings, and the ranking number of a graph is trivially equal to its ``tree-depth,'' a term introduced by Ne\u{s}et\u{r}il and Ossona de Mendez in 2006 \cite{NO} in developing their theory of graph classes having bounded expansion.

In general, rankings are used to design efficient divide-and-conquer strategies for minimizing the time needed to perform interrelated tasks in parallel \cite{IRV1}.  The most basic example concerns a complex product being assembled in stages from its individual parts, where each stage of construction consists of individual parts being attached to previously assembled components in such a way that no component ever has more than one new part.  Here, the complex product is represented by the graph $G$ whose vertices are the individual parts and whose edges are the connections between those parts; assuming all parts require the same amount of to be installed, the fewest number of stages needed to complete construction is $\rho(G)$, achieved by finding some $\rho(G)$-ranking $\alpha$ of $G$ and installing each part $v$ in stage $\alpha (v)$.  Similarly, rankings can be use to optimize the disassembly of a product into parts, where each stage of deconstruction consists of individual parts being detached from remaining components in such a way that no component loses multiple parts at the same time; $G$ can be disassembled in $\rho (G)$ stages by removing each part $v$ in stage $\rho(G)-\alpha(v)+1$.  

The vertex ranking problem has spawned multiple variations, including edge ranking \cite{IRV2}, on-line ranking \cite{BH,M}, and list ranking, studied here.  The list ranking problem is to vertex ranking as the list coloring problem is to ordinary vertex coloring. 
\begin{definition}
A function $L$ that assigns each vertex of $G$ a finite set of positive integers is an \emph{$f$-list assignment} for $G$ if $|L(v)|=f(v)$ for each $v\in V(G)$.  If $|L(v)|=k$ for all $v$, then $L$ is \emph{$k$-uniform}.  An \emph{$L$-ranking} of $G$ is a ranking $\alpha$ such that $\alpha(v)\in L(v)$ for each $v$.  Say that $G$ is \emph{$f$-list rankable} if $G$ has an $L$-ranking whenever $L$ is an $f$-list assignment for $G$, and say that $G$ is \emph{$k$-list rankable} if $G$ is \emph{$f$-list rankable} for $f=k$.  Let the \emph{list ranking number} of $G$, denoted $\rho _{\ell}(G)$, be the least $k$ such that $G$ is $k$-list rankable.
\end{definition}
Note that $G$ is $f$-rankable if $G$ is $f$-list rankable, since an $f$-ranking is just an $L$-ranking if $L$ is the $f$-list assignment defined by $L(v)=\{1,\ldots,f(v)\}$ for each $v\in V(G)$.  In terms of our application concerning a product made from parts, list ranking corresponds to determining the feasability of (dis)assembly when predetermined scheduling restraints limit the times at which each individual part can be (dis)assembled.  The complexity of finding rankings from a given list assignment was considered by Dereniowski in 2008 \cite{D}.

In Section \ref{sec:games} we introduce three on-line versions of list ranking, which relate to list ranking similarly to the way on-line list coloring (also known as paintability) relates to list coloring \cite{S}.  We define these on-line versions of list ranking as games between adversaries Taxer and Ranker.  At the beginning of the game each vertex is assigned a size for its list of potential labels, but no actual labels.  Taxer in effect fills out these lists in real time by iterating through the possible labels one by one, stipulating which vertices have the given label in their lists (the order in which Taxer iterates through the labels depends on the version of the game, and once a label comes up it cannot be revisited).  Ranker must decide immediately which of those vertices just selected by Taxer are to receive the given label, extending a partial ranking of the graph (Taxer is allowed to use knowledge of the partial ranking already created by Ranker when deciding which vertices are to have a given label in their lists).  Ranker wins by creating a ranking of the graph before any vertex has its list filled with unused labels, and Taxer wins otherwise.  

Complete descriptions of the on-line ranking games are postponed until Section \ref{sec:games}, though we do introduce relevant terminology and notation here.  Given a graph $G$ and function $f:V(G)\rightarrow\mathbb{N}$, we say $G$ is \emph{on-line $f$-list rankable} if Ranker has a winning strategy for the game $R^{\pm}(G,f)$, and we define the \emph{on-line list ranking number} of $G$, denoted $\rho ^{\pm}_{\ell}(G)$, to be the least $k$ such that $G$ is on-line $f$-list rankable for the constant function $f=k$.  We similarly define the \emph{on-line list low-ranking number} $\rho ^-_{\ell}(G)$ and \emph{on-line list high-ranking number} $\rho ^+_{\ell}(G)$ based on the games $R^-(G,f)$ and $R^+(G,f)$.

Just as we have seen that finding an $f$-ranking of $G$ is equivalent to finding an $L$-ranking of $G$ for a special $f$-list assignment $L$, we will see that any $f$-list assignment can be modeled by special strategies for Taxer in the games $R^-(G,f)$ and $R^+(G,f)$, and these games in turn can be modeled by special Taxer strategies in $R^{\pm}(G,f)$.  Thus our parameters will satisfy $\rho(G)\leq \rho _{\ell}(G)\leq \min\{\rho ^-_{\ell}(G),\rho ^+_{\ell}(G)\}\leq\max\{\rho ^-_{\ell}(G),\rho ^+_{\ell}(G)\}\leq \rho ^{\pm}_{\ell}(G)$.

In Section \ref{sec:prelim} we investigate how ranking a graph relates to ranking its minors in our various versions of the ranking problem.

In Sections \ref{sec:paths} and \ref{sec:cycles} we consider paths and cycles.  As stated in \cite{BDJKKMT}, $\rho(P_n)=\left\lceil \log (n+1)\right\rceil$: the largest label can appear only once, so rankings are achieved by individually ranking the subpaths on either side of the vertex receiving the largest label.  For instance, $P_7$ can be ranked by labeling vertices from left to right with $1,2,1,3,1,2,1$.  Furthermore, $\rho(C_n)=1+\left\lceil \log n\right\rceil$, as stated in \cite{BH}: the largest label can appear only once, so rankings are achieved by ranking the copy of $P_{n-1}$ left by deleting the vertex receiving the largest label.

The first main result of this paper is proved in Section \ref{sec:paths}.
\begin{thm}\label{IntroPath}
$\rho ^{\pm}_{\ell}(P_n)=\rho(P_n)$.
\end{thm}
This theorem is proved using a more general result.  For a nonnegative integer valued function $f$ whose domain includes some finite set $V$ of vertices, define $\sigma_f(V)=\sum_{v\in V}2^{-f(v)}$; we prove that $P_n$ is on-line $f$-list rankable if $\sigma_f(V(P_n))<1$.  The result is sharp, because $P_{2^k}$ is not even $k$-rankable; the construction extends to all $n$.  On the other hand, for $n\geq 5$ there are functions $f$ satisfying $\sigma_f(V(P_n))>1$ such that $P_n$ is on-line $f$-list rankable.

The second main result of this paper is proved in Section \ref{sec:cycles} and relies heavily on the work of the previous section.
\begin{thm}\label{IntroCycle}
$\rho ^{\pm}_{\ell}(C_n)=\rho(C_n)$.
\end{thm}
Turning our attention toward more complicated graphs, we note the following lower bound for $\rho_{\ell}(G)$.  
\begin{prop}\label{TreeBound}
If $q$ is the maximum number of leaves in a subtree $T$ of a graph $G$, then $\rho _{\ell}(G)\geq q$.
\end{prop}
\begin{proof}
Construct a $(q-1)$-uniform list assignment $L$ by giving each vertex that is not a leaf of $T$ the list $\{1,\ldots,q-1\}$ and each leaf of $T$ the list $\left\{ q,\ldots ,2q-2\right\}$.  If $G$ has an $L$-ranking, then two leaves $u$ and $v$ of $T$ must receive the same label, but no interior vertex of the $u,v$-path through $T$ can contain a larger label.
\end{proof}
Thus $\rho _{\ell}(G)-\rho(G)$ can be made arbitrarily large, since $\rho(K_{1,n})=2$ (label the leaves with $1$ and the center with $2$) but $\rho _{\ell}(K_{1,n})\geq n$.  Section ~\ref{sec:trees} finds a class of trees for which the bound of Proposition \ref{TreeBound} is equality, yielding our third main result.
\begin{thm}\label{IntroTree}
For any positive integer $p$, there is a positive integer $q_p$ such that for any tree $T$ with $p$ internal vertices and at least $q_p$ leaves, $\rho _{\ell}(T)$ equals the number of leaves of $T$.
\end{thm}
The proof Theorem \ref{IntroTree} gives a specific value of $q_p$ that is exponential in $p$, but this does not appear to be anywhere near sharp.
\begin{con}
If $T$ is a tree with $p$ internal vertices and $q$ leaves, and $p<q$, then $\rho _{\ell}(T)=q$.
\end{con}
It also seems likely that Theorem \ref{IntroTree} can be strengthened in another way.
\begin{con}
For any positive integer $p$, there is a positive integer $q_p$ such that for any tree $T$ with $p$ internal vertices and at least $q_p$ leaves, $\rho^{\pm}_{\ell}(T)$ equals the number of leaves of $T$.
\end{con}
\section{On-line list ranking games}
\label{sec:games}
In this section we introduce the various on-line versions of list ranking, presented as games between Taxer and Ranker.  Let $G$ be a graph and $f:V(G)\rightarrow\mathbb{N}$.  Each $v\in V(G)$ starts the game possessing $f(v)$ tokens (representing the open slots in the list of labels available to $v$), and each round of the game corresponds to a label to be used by Ranker to rank $G$.  Taxer starts the round corresponding to the label $c$ by taking tokens from a set $T$ of unranked vertices in $G$ (in effect inserting $c$ into the list of each vertex in $T$).  Ranker responds by removing some subset $R$ of $T$ from $G$ and potentially adding certain edges to the resulting graph (in effect labeling vertices in $R$ with $c$, with restrictions placed on $R$ according to $c$, and the set of edges added determined by $c$ and $R$).  Taxer's goal is to bankrupt some vertex without Ranker being able to assign it any of its possible labels.

The three on-line list ranking games we introduce differ in the order the potential labels are introduced.  We start with the game in which the labels are introduced in increasing order.  In terms of the application of list ranking given in Section \ref{sec:intro} concerning a product made from parts, this variation corresponds to determining the feasability of assembly when the list of individual parts that can be attached at each stage of construction is not known until that stage is reached.
\begin{definition}
$G$ is \emph{on-line $f$-list low-rankable} if Ranker has a winning strategy over Taxer in the following game $R^-(G,f)$, which starts by setting $G_1=G$ and allotting $f(v)$ tokens to each vertex $v$.  During round $i$ for $i\geq 1$, Taxer begins play by taking a single token from each element of a nonempty set $T_i$ of vertices of $G_i$.  Ranker responds by creating the graph $G_{i+1}$ from $G_i$ by first selecting a subset $R_i$ of $T_i$ that is independent in $G_i$, then adding an edge between any nonadjacent vertices that have a common neighbor in $R_i$, and finally deleting $R_i$.  Taxer wins after round $i$ if $G_{i+1}$ contains any vertex without a token; Ranker wins after round $i$ if $G_{i+1}$ is empty.  Say $G$ is \emph{on-line $k$-list low-rankable} if $G$ is on-line $f$-list low-rankable for the constant function $f=k$.  Let the \emph{on-line list low-ranking number} of $G$, denoted $\rho ^-_{\ell}(G)$, be the least $k$ such that $G$ is on-line $k$-list low-rankable.  
\end{definition}
Note that if Ranker wins the game after round $j$, then $G$ can be given a $j$-ranking by labeling each vertex in $R_i$ with $i$: if $u,v\in R_i$ then $u$ and $v$ are not adjacent in $G_i$, so each $u,v$-path in $G$ has an interior vertex $w$ that appears in $G_i$ but not $R_i$, meaning $w$ receives a larger label than do $u$ and $v$.  Further note that $G$ is $f$-list rankable if $G$ is on-line $f$-list low-rankable: finding an $L$-ranking from an $f$-list assignment $L$, whose set of labels we may presume to be precisely $\{1,\ldots,m\}$, is equivalent to finding a winning strategy for Ranker in $R^-(G,f)$ where Taxer must declare before the game that any vertex $v$ remaining in round $i$ will be put in $T_i$ if and only if $i\in L(v)$.

\begin{figure}[ht]
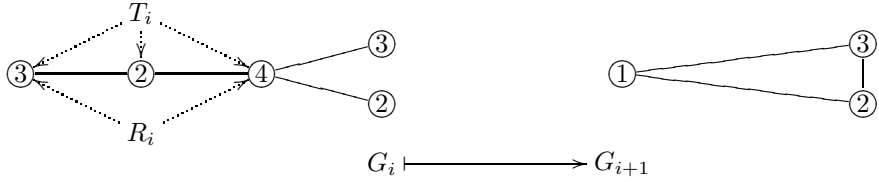

\centering
\[ \xygraph{
!{<0cm,0cm>;<1.6cm,0cm>:<0cm,.8cm>::}
!{(0,0) }*+[o][F-:<3pt>]{3}="a"
!{(1,0) }*+[o][F-:<3pt>]{2}="b"
!{(2,0) }*+[o][F-:<3pt>]{4}="c"
!{(3,-.5) }*+[o][F-:<3pt>]{2}="d"
!{(3,.5) }*+[o][F-:<3pt>]{3}="e"
!{(5,0) }*+[o][F-:<3pt>]{1}="f"
!{(7,-.5) }*+[o][F-:<3pt>]{2}="g"
!{(7,.5) }*+[o][F-:<3pt>]{3}="h"
!{(1,1) }*+{T_i}="t"
!{(1,-1) }*+{R_i}="r"
!{(3,-1.5) }*+{G_i}="p"
!{(5,-1.5) }*+{G_{i+1}}="q"
"a"-"b"-"c"-"d" "c"-"e"
"t"-@{.>}"a" "t"-@{.>}"b" "t"-@{.>}"c"
"r"-@{.>}"a" "r"-@{.>}"c"
"f"-"g"-"h"-"f"
"p"-@{|->}"q"
} \]
\caption{A possibility for round $i$ of $R^-(G,f)$ (or of $R^{\pm}(G,f)$, if the round is low), where the number on each vertex counts its remaining tokens.}
\end{figure}
We now introduce the on-line list ranking game in which the labels are introduced in decreasing order (should Ranker win after round $j$, the actual value of the first label introduced will be $j$).  In terms of our application of list ranking, this variation corresponds to determining the feasability of dissassembly when the list of individual parts that can be detached at each stage of deconstruction is not known until that stage is reached.
\begin{definition}
$G$ is \emph{on-line $f$-list high-rankable} if Ranker has a winning strategy over Taxer in the following game $R^+(G,f)$, which starts by setting $G_1=G$ and allotting $f(v)$ tokens to each vertex $v$.  During round $i$ for $i\geq 1$, Taxer begins play by taking a single token from each element of a nonempty set $T_i$ of vertices of $G_i$.  Ranker responds by creating the induced subgraph $G_{i+1}$ of $G_i$ by deleting from $G_i$ a subset $R_i$ of $T_i$ no two vertices of which lie in the same component.  Taxer wins after round $i$ if $G_{i+1}$ contains any vertex without a token; Ranker wins after round $i$ if $G_{i+1}$ is empty.  Say $G$ is \emph{on-line $k$-list high-rankable} if $G$ is on-line $f$-list high-rankable for the constant function $f=k$.  Let the \emph{on-line list high-ranking number} of $G$, denoted $\rho ^+_{\ell}(G)$, be the least $k$ such that $G$ is on-line $k$-list high-rankable.  
\end{definition}
Note that if Ranker wins the game after round $j$, then $G$ can be given a $j$-ranking by labeling each vertex in $R_i$ with $j+1-i$: if $u,v\in R_i$ then $u$ and $v$ are in different components of $G_i$, so each $u,v$-path in $G$ has an interior vertex $w$ that does not make it to $G_i$, meaning $w$ receives a larger label than do $u$ and $v$.  Further note that $G$ is $f$-list rankable if $G$ is on-line $f$-list high-rankable: finding an $L$-ranking from an $f$-list assignment $L$, whose set of labels we may presume to be precisely $\{1,\ldots,m\}$, is equivalent to finding a winning strategy for Ranker in $R^+(G,f)$ where Taxer must declare before the game that any vertex $v$ remaining in round $i$ will be put in $T_i$ if and only if $m+1-i\in L(v)$.

\begin{figure}[hb]
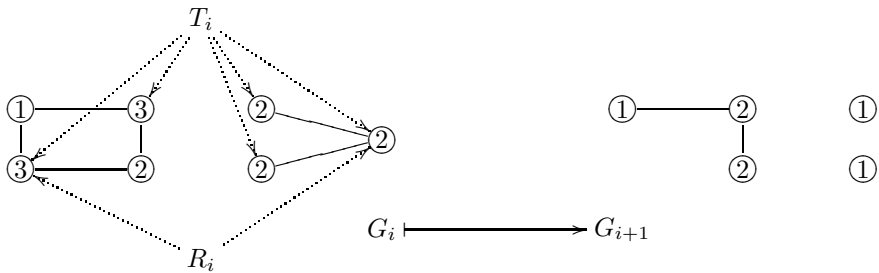

\centering
\[ \xygraph{
!{<0cm,0cm>;<1.6cm,0cm>:<0cm,.8cm>::}
!{(0,.5) }*+[o][F-:<3pt>]{1}="a"
!{(1,.5) }*+[o][F-:<3pt>]{3}="b"
!{(0,-.5) }*+[o][F-:<3pt>]{3}="x"
!{(1,-.5) }*+[o][F-:<3pt>]{2}="y"
!{(3,0) }*+[o][F-:<3pt>]{2}="c"
!{(2,-.5) }*+[o][F-:<3pt>]{2}="d"
!{(2,.5) }*+[o][F-:<3pt>]{2}="e"
!{(5,.5) }*+[o][F-:<3pt>]{1}="u"
!{(6,.5) }*+[o][F-:<3pt>]{2}="v"
!{(6,-.5) }*+[o][F-:<3pt>]{2}="f"
!{(7,-.5) }*+[o][F-:<3pt>]{1}="g"
!{(7,.5) }*+[o][F-:<3pt>]{1}="h"
!{(1.5,2) }*+{T_i}="t"
!{(1.5,-2) }*+{R_i}="r"
!{(3,-1.5) }*+{G_i}="p"
!{(5,-1.5) }*+{G_{i+1}}="q"
"a"-"b"-"y"-"x"-"a" "c"-"d" "c"-"e"
"t"-@{.>}"b" "t"-@{.>}"x" "t"-@{.>}"c" "t"-@{.>}"d" "t"-@{.>}"e"
"r"-@{.>}"x" "r"-@{.>}"c"
"u"-"v"-"f"
"p"-@{|->}"q"
} \]
\caption{A possibility for round $i$ of $R^+(G,f)$ (or of $R^{\pm}(G,f)$, if the round is high), where the number on each vertex counts its remaining tokens.}
\end{figure}
Our final on-line list ranking game is basically a mixture of the first two, in that at the beginning of each round Taxer gets to decide whether the label to be assigned to vertices that round is either the least or greatest label yet to be used.  In terms of our application of list ranking, this variation corresponds to determining the feasability of (dis)assembly when each stage consists of learning either the list of individual parts that can be (dis)assembled immediately or the list of individual parts that can be scheduled to be (dis)assembled at the same time after the (dis)assembly of all individual parts that have not yet been scheduled for future (dis)assembly.
\begin{definition}
$G$ is \emph{on-line $f$-list rankable} if Ranker has a winning strategy over Taxer in the following game $R^{\pm}(G,f)$, which starts by setting $G_1=G$ and allotting $f(v)$ tokens to each vertex $v$.  During round $i$ for $i\geq 1$, Taxer begins play by declaring the round to be either low or high; low rounds are played like rounds of $R^-(G,f)$ and high rounds are played like rounds of $R^+(G,f)$.  As in $R^-(G,f)$ and $R^+(G,f)$, Taxer wins $R^{\pm}(G,f)$ after round $i$ if $G_{i+1}$ contains any vertex without a token, and Ranker wins after round $i$ if $G_{i+1}$ is empty.  Say $G$ is \emph{on-line $k$-list rankable} if $G$ is on-line $f$-list rankable for the constant function $f=k$.  Let the \emph{on-line list ranking number} of $G$, denoted $\rho ^{\pm}_{\ell}(G)$, be the least $k$ such that $G$ is on-line $k$-list rankable.   
\end{definition}
Note that if Ranker wins the game after round $j$, then $G$ can be given a $j$-ranking by labeling each vertex in $R_i$ with $i'$ if round $i$ was the $i'$th low round and with $j+1-i''$ if round $i$ was the $i''$th high round.  Further note that $G$ is on-line $f$-list low-rankable and on-line $f$-list high-rankable if $G$ is on-line $f$-list rankable: a winning strategy for Ranker in $R^-(G,f)$ is a winning strategy for Ranker in $R^{\pm}(G,f)$ where Taxer is required to declare each round low, and a winning strategy for Ranker in $R^+(G,f)$ is a winning strategy for Ranker in $R^{\pm}(G,f)$ where Taxer is required to declare each round high.

We summarize our observations about the relationships among the parameters we have introduced.
\begin{prop}\label{Comparison}
For any graph $G$, we have $\rho (G)\leq \rho _{\ell}(G)\leq \min\{\rho ^-_{\ell}(G),\rho ^+_{\ell}(G)\}\leq\linebreak\max\{\rho ^-_{\ell}(G),\rho ^+_{\ell}(G)\}\leq \rho ^{\pm}_{\ell}(G)$.
\end{prop}
Currently we have no example of a graph $G$ satisfying $\rho _{\ell}(G)<\rho ^{\pm}_{\ell}(G)$; it would be especially interesting to find a construction that could make $\rho ^{\pm}_{\ell}(G)-\rho _{\ell}(G)$ arbitrarily large.  Furthermore, we know of no $G$ and function $f$ such that $G$ is on-line $f$-list high-rankable but not on-line $f$-list rankable.  We can, however, present a graph $G$ and function $f$ such that $G$ is on-line $f$-list low-rankable but not on-line $f$-list high-rankable.  We use a lemma (which we will use several times more in later sections) that provides list ranking and on-line list ranking analogues of the following observation.  If the vertices of a graph $G$ can be labeled $v_1,\ldots,v_n$ so that for some index $k$, the subgraph $G'$ of $G$ induced by $\{v_1,\ldots,v_k\}$ is $f$-rankable, and $f(v_i)\geq i$ for $k<i\leq n$, then we can construct an $f$-ranking of $G$ by giving $G'$ a $k$-ranking (which is possible since at most $k$ labels can be used in a ranking of $G'$) and labeling $v_i$ with $i$ for $k<i\leq n$.  
\begin{lem}\label{Distinct}
Let $G$ be a graph with vertices $v_1,\ldots,v_n$, and for some index $k$ let $G'$ be the subgraph of $G$ induced by $\{v_1,\ldots,v_k\}$.  Suppose that for every component $C$ of $G-V(G')$, the set of vertices in $G'$ adjacent to vertices in $C$ is a (possibly empty) clique.  Let $f:V(G)\rightarrow\mathbb{N}$ satisfy $f(v_i)\geq i$ for $k<i\leq n$.  If $G'$ is $f$-list rankable, then so is $G$, and for $*\in\{-,+,\pm\}$, if Ranker wins $R^*(G',f)$, then Ranker also wins $R^*(G,f)$.
\end{lem}
\begin{proof}
If $G'$ is $f$-list rankable and $L$ is an $f$-list assignment for $G$, then we can create an $L$-ranking of $G$ in the following manner.  First rank $G'$ from $L$, then delete the at most $k$ labels used on $G'$ from the remaining lists, and finally label $v_{k+1},\ldots,v_n$ in order from their remaining lists, deleting the label given to $v_i$ from the list of each unlabeled vertex.  To see that this completes a ranking of $G$, we note that any path $P$ in $G$ between vertices with the same label must have endpoints in $G'$, in which case by hypothesis $P$ could be modified into a path $P'$ in $G'$ by replacing each maximal subpath of $P$ in $G-V(G')$ with an edge in $G'$.  Since $G'$ is ranked, and the endpoints of $P'$ have the same label, an internal vertex of $P'$ must contain a larger label, and this vertex is also an internal vertex of $P$.

We now prove by induction on $n$ that if Ranker has a winning strategy on $R^*(G',f)$ for some $*\in\{-,+,\pm\}$, then Ranker can win $R^*(G,f)$.  The base case of $n=1$ is trivial, so we assume $n>1$ and that the statement holds if $G$ has fewer than $n$ vertices.  Let $T\subseteq V(G)$ be the set of vertices from which Taxer takes a token in the first round of $R^*(G,f)$, let $R\subseteq T$ be the set of vertices to be removed by Ranker in response, let $H$ be the graph to be played on in the second round (determined by $R$ and whether the first round was high or low), and let $H'$ be the subgraph of $H$ induced by $V(G')-R$.  

Let $h:V(H)\rightarrow\mathbb{N}$ be defined by $h(v_i)=f(v_i)-|T\cap\{v_i\}|$ and $g:\mathbb{N}\rightarrow\mathbb{N}$ be defined by $g(i)=i-|R\cap\{v_1,\ldots,v_{i-1}\}|$.  We complete the proof by showing Ranker can always choose $R$ so that $h(v_i)\geq g(i)$ for $v_i\in V(H)-V(H')$, Ranker has a winning strategy on $R^*(H',h)$, and the set $S$ of vertices in $H'$ adjacent to any component $C$ of $H-V(H')$ is a (possibly empty) clique.  Fix such a component $C$, and define $S$ as above. 

If $T\cap V(G')=\emptyset$, let $R=\{v_j\}$, where $j$ is the least index of any vertex in $T$ (note that $j>k$); clearly this is a legal move by Ranker.  In this case, $H'=G'$, $h(v_i)=f(v_i)$ for $1\leq i<j$ (with $f(v_i)\geq g(v_i)$ for $k<i<j$), and $h(v_i)\geq f(v_i)-1=g(i)$ for $j<i\leq n$.  Since Ranker has a winning strategy on $R^*(G',f)$, Ranker also has a winning strategy on $R^*(H',h)$.  Since all vertices in $S$ are adjacent in $G$ to vertices in the component of $G-V(G')$ containing $C$, $S$ is a subset of a clique and thus a clique itself. 

If $T\cap V(G')\neq\emptyset$, let $R$ consist of the vertices that Ranker would remove in a winning strategy on $R^*(G',f)$ in response to Taxer removing tokens from $T\cap V(G')$.  This is a legal move by Ranker because if the first round is low, then $R$ is independent in $G'$ and thus also in $G$, and if the first round is high, then any two vertices in $R$ are in different components of $G'$ and thus also in different components of $G$ (since otherwise some component of $G-V(G')$ would have nonadjacent neighbors in $G'$).  Clearly, $h(v_i)\geq f(v_i)-1\geq g(i)$ for $v_i\in V(H)-V(H')$, Ranker has a winning strategy on $R^*(H',h)$, and $C$ is a component of $G-V(G')$.  We need only show $S$ is a clique.  

If the first round is high, then $H=G-R$, so all vertices in $S$ are adjacent in $G$ to vertices in $C$, making $S$ a subset of a clique and thus a clique itself.  If the first round is low, then $H$ is obtained from $G$ by deleting each vertex in $R$ after completing its neighborhood.  Any vertex in $S$ is either adjacent in $G$ to a vertex in $C$ or adjacent in $G$ to a vertex in $R$ adjacent to a vertex in $C$, so $S$ is a clique since the set of vertices in $G'$ adjacent to vertices in $C$ is a clique, and the neighborhood of any vertex in $R$ is completed when forming $H$ from $G$.  
\end{proof}
\begin{prop}
If the vertices of $P_n$ are $v_1,\ldots,v_n$ in order, then $P_n$ is on-line $f$-list low-rankable but not on-line $f$-list high-rankable for $f=(3,1,2,3,5,6,\ldots,n)$.
\end{prop}
\begin{proof}
We show $P_n$ is on-line $f$-list low-rankable by giving a winning strategy for Ranker on $R^-(P_n,f)$.  By Lemma \ref{Distinct} it suffices to exhibit a winning strategy for Ranker on $R^-(P_4,(3,1,2,3))$.  If Taxer selects $v_2$ in the first round, let Ranker remove $v_2$, and also remove $v_4$ if Taxer selects that as well.  Then, assuming Taxer also removed tokens from $v_1$ and $v_3$, the game reduces to either $R^-(P_2,(2,1))$ ($v_2$ and $v_4$ removed by Ranker) or $R^-(P_3,(2,1,3))$ ($v_2$ removed), both of which lead to victory for Ranker, by Lemma \ref{Distinct}.  

If Taxer does not select $v_2$ in the first round, let Ranker remove $v_1$ if selected, $v_3$ if selected, and $v_4$ if selected and $v_3$ is not selected.  Then, assuming Taxer removed a token from $v_4$ if one was also taken from $v_3$, the game reduces to either $R^-(P_2,(1,2))$ ($v_1$ and either $v_3$ or $v_4$ removed), $R^-(P_3,(1,2,3))$ ($v_1$ removed), or $R^-(P_3,(3,1,2))$ (either $v_3$ or $v_4$ removed), each of which leads to victory for Ranker, by Lemma \ref{Distinct}.

We show $P_n$ is not on-line $f$-list high-rankable by giving a winning strategy for Taxer on $R^+(P_n,f)$.  Let Taxer begin play by selecting $v_1$ and $v_4$.  If Ranker responds by removing $v_1$, let Taxer next select $v_2$, $v_3$, and $v_4$, leaving $v_2$ with no tokens and $v_3$ and $v_4$ with one each.  Then Ranker must remove $v_2$, and if Taxer selects $v_3$ and $v_4$, Ranker can remove just one of these.  The other vertex is left with no token, so Taxer wins.  

If Ranker responds to Taxer's initial selection by removing $v_4$, let Taxer next select $v_1$ and $v_3$, leaving each remaining vertex with a single token.  Ranker cannot remove both $v_1$ and $v_3$, leaving $v_2$ adjacent to another vertex, each with a single token.  If Taxer selects both of these vertices then Ranker can remove just one, leaving the other with no token.  Thus Taxer wins $R^+(P_n,f)$.
\end{proof}
\section{List ranking and on-line list ranking graph minors}
\label{sec:prelim}
We now examine how ranking a graph relates to ranking one of its minors in our various versions of the ranking problem.  We first recall the definition of a graph minor and introduce some notation to be used in this section.
\begin{definition}
To \emph{contract} an edge $uv$ from a graph $G$, delete $u$ and $v$ and add a new vertex $w$ adjacent to all former neighbors of $u$ and $v$ in $G$.  A \emph{minor} of $G$ is any graph $G'$ that can be obtained by performing zero or more edge contractions on a subgraph of $G$.
\end{definition}
For the rest of this section, fix a graph $G$, a minor $G'$ of $G$, and a function $f:V(G)\cup V(G')\rightarrow\mathbb{N}$.  
\begin{definition}
For $w\in V(G')$, define $U(w)=\{w\}$ if $w\in V(G)$ and otherwise define $U(w)$ to be the set consisting of each $u\in V(G)$ such that a series of contractions turned an edge containing $u$ into $w$.  Let $H_w$ denote the subgraph of $G$ induced by $U(w)$.
\end{definition}
Note that if $w\in V(G')$, then $H_w$ is a connected subgraph of $G$.  Furthermore, if $x$ and $y$ are disjoint vertices in $G'$, then $U(x)\cap U(y)=\emptyset$, but $U(x)\cup U(y)$ induces a connected subgraph of $G$ if $x$ and $y$ are adjacent in $G'$.
\begin{prop}\label{LMinor}
If $G$ is $f$-list rankable, but $H_w$ is not $(f-f(w))$-list rankable for each $w\in V(G')$, then $G'$ is $f$-list rankable.  
\end{prop}
\begin{proof}
We show $G'$ is $f'$-list rankable by constructing an $L'$-ranking $\alpha'$ from an arbitrary $f$-list assignment $L'$ for $G'$.  For each $w\in V(G')-V(G)$, let $L_w$ be an $(f-f(w))$-list assignment for $H_w$ such that $H_w$ has no $L_w$-ranking; without loss of generality assume that the smallest label in any list assigned by $L'$ is larger than the largest label in any list assigned by any $L_w$.  Let $L$ be the $f$-list assignment for $G$ obtained from $L'$ by letting $L(u)=L'(w)\cup L_w(u)$ for each $w\in V(G')-V(G)$ and $u\in U(w)$, $L(v)=L'(v)$ for each $v\in V(G)\cap V(G')$, and $L(t)=\{1,\ldots,f(t)\}$ for each $t\in V(G)-\bigcup_{w\in V(G')}U(w)$.  

By hypothesis $G$ has an $L$-ranking $\alpha$.  We modify $\alpha$ into an $L'$-ranking $\alpha'$ of $G'$ by letting $\alpha'(w)=\max\{\alpha(u):u\in U(w)\}$.  Note that $\alpha'(w)\in L'(w)$ for all $w\in V(G')$ since $\alpha$ cannot assign every vertex in $H_w$ a label from $L_w$ (or else $H_w$ would have an $L_w$-ranking), and the smallest label in $L'(w)$ is larger than the largest label anywhere in $L_w$.  We prove that $\alpha'$ is in fact a ranking of $G'$ by showing that if $w_1w_2\ldots w_m$ is a nontrivial path $P'$ in $G'$ satisfying $\alpha'(w_1)=\alpha'(w_m)$, then $\alpha'(w_i)>\alpha'(w_1)$ for some $1<i<m$. 

By the definition of $\alpha'$, for $1\leq i\leq m$ there exists $u_i\in U(w_i)$ such that $\alpha (u_i)=\alpha'(w_i)$.  Since $w_i$ and $w_{i+1}$ are adjacent in $G'$ for $1\leq i<m$, for $1\leq i\leq m$ there exist (not necessarily distinct) vertices $x_i,y_i\in U(w_i)$ such that $x_1=u_1$, $y_m=u_m$, and $y_i$ is adjacent to $x_{i+1}$ in $G$ for $1\leq i<m$.  Since each $H_{w_i}$ is connected, there exists a $x_i,y_i$-path $P^i$ in $H_{w_i}$ for $1\leq i\leq m$.  Hence concatenating the paths $P^1,P^2,\ldots,P^m$ forms a $u_1,u_m$-path $P$ in $G$.  Since $\alpha$ is a ranking of $G$ and $\alpha(u_1)=\alpha'(w_1)=\alpha'(w_m)=\alpha(u_m)$, for some $i$ there exists $z\in V(P^i)$ satisfying $\alpha(z)>\alpha(u_1)$, in which case $\alpha'(w_i)\geq\alpha(z)>\alpha(u_1)=\alpha'(w_1)$.
\end{proof}
\begin{prop}\label{-Minor}
Fix $*\in\{-,\pm\}$.  If Ranker wins $R^*(G,f)$ but Taxer wins $R^-(H_w,f-f(w))$ for each $w\in V(G')$, then Ranker wins $R^*(G',f)$.
\end{prop}
\begin{proof}
Based on Taxer's strategy on $R^*(G',f)$, we define a strategy for Taxer on an auxilliary game $R^*(G,f)$, and use Ranker's winning strategy on $R^*(G,f)$ to define a winning strategy for Ranker on $R^*(G',f)$.  Let Taxer begin the auxilliary game $R^*(G,f)$ by isolating play to each $H_w$ one-at-a-time and copying a winning strategy from $R^-(H_w,f-f(w))$ until some $u\in U(w)$ is left with at most $f(w)$ tokens; once this happens say that $w$ and $u$ are \emph{partners}.  Let Taxer continue the auxilliary game $R^*(G,f)$ by declaring low rounds and taking tokens from vertices of $G$ not partnered with vertices of $G'$ until Ranker has removed all such vertices.  

Since each round so far has been low, the neighborhood of each vertex removed by Ranker has been completed before the next round, so the partnership between the vertices of $G'$ and the vertices of the altered graph $H$ of $R^*(G,f)$ is a graph isomorphism from $G'$ to a spanning subgraph of $H$.  Letting $h(u)$ count the tokens on each $u\in V(H)$, we note that $h(u)\leq f(w)$ if $u$ is partnered with $w\in V(G')$.  The auxilliary game $R^*(G,f)$ has been reduced to $R^*(H,h)$, for which Ranker has a winning strategy since Ranker wins $R^*(G,f)$.  

To complete the proof, we use induction on $|V(H)|$ to show that if $G'$ is isomorphic to a subgraph of $H$, with Ranker having a winning strategy for $R^*(H,h)$ and $f(w)\geq h(u)$ for each $w\in V(G')$ partnered with $u\in V(H)$, then Ranker wins $R^*(G',f)$.  The base case of $H=K_1$ is trivial, so we may assume that $|V(H)|>1$ and the statement holds for smaller graphs.  Supposing Taxer begins $R^*(G',f)$ by taking tokens from vertices in the set $T'\subseteq V(G')$, let Taxer declare the first round of $R^*(H,h)$ to be the same type before selecting the set $T$ of vertices of $H$ partnered with vertices in $T'$.  If $R$ is the set of vertices removed by Ranker to create the graph $F$ in a winning strategy on $R^*(H,h)$, let Ranker respond in $R^*(G',f)$ by removing the set $R'$ of vertices of $G'$ partnered with vertices in $R$ to create the graph $F'$.  

Clearly every element of $V(F')$ has a partner in $V(F)$ with at most as many tokens remaining, and the partnership relation still provides a graph isomorphism from $F'$ to a subgraph of $F$ whether the first round was low (vertices selected by Ranker are deleted but their neighborhoods are completed) or high (vertices selected by Ranker are simply deleted without adding any new edges).  Since Ranker wins $R^*(H,h)$, Ranker can win the new game on $F$, so by the inductive hypothesis Ranker can win the new game on $F'$ as well. Thus Ranker wins $R^*(G',f)$.
\end{proof}
\begin{prop}\label{+Minor}
If Ranker wins $R^+(G,f)$ but Taxer wins $R^+(H_w,f-f(w))$ for each $w\in V(G')$, then Ranker wins $R^+(G',f)$.
\end{prop}
\begin{proof}
We show Ranker wins $R^+(G',f)$ by performing induction on $|V(G)|$, with the base case of $G=K_1$ being trivial.  Now assume $|V(G)|>1$ and the statement holds for smaller graphs.  Suppose Taxer begins $R^+(G',f)$ by taking tokens from vertices in the set $T'\subseteq V(G')$.  Based on $T'$, we decide from which set $T\subseteq V(G)$ Taxer will take tokens in the first round of the auxilliary game $R^+(G,f)$: let $T=\bigcup_{w\in T'}U(w)$.  After Ranker responds as part of a winning strategy in the auxilliary game $R^+(G,f)$ by removing the vertices of some $R\subseteq T$ to create the graph $F$, we decide which set $R'\subseteq T'$ Ranker will remove from $G'$ to create the graph $F'$ in $R^+(G',f)$: let $R'=\{w\subseteq T':U(w)\cap R\neq\emptyset\}$.  

Then $R^+(G',f)$ has been reduced to $R^+(F',h)$ and $R^+(G,f)$ has been reduced to $R^+(F,h)$, where the function $h$ is defined by $h(v)=f(v)-1$ if either $v\in V(G')$ and $U(v)\subseteq T-R$ or $v\in V(G)$ and $v\in T-R$, and $h=f$ elsewhere.  To complete the proof, we show that Ranker wins $R^+(F',h)$.  Note that $F'$ is a minor of $F$ (since $w\in V(F')$ if and only if $U(w)\subseteq V(F)$), Ranker wins $R^+(F,h)$ (since Ranker wins $R^+(G,f)$), and $|V(F)|<|V(G)|$.  Thus by the inductive hypothesis it suffices to show for each $w\in V(F')$ that $H_w$ remains intact as a subgraph of $F$ and Taxer wins $R^+(H_w,h-h(w))$.

We need only consider $w\notin R'$, since $w\in V(F')$ if and only if $w\notin R'$.  If $w\notin T'$, then $U(w)\cap T=\emptyset$, so $h(v)=f(v)$ for $v\in \{w\}\cup U(w)$.  If $w\in T'-R'$, then $U(w)\subseteq T-R$, so $h(v)=f(v)-1$ for $v\in\{w\}\cup U(w)$.  In either case $H_w$ is a subgraph of $F$ and $h-h(w)=f-f(w)$.  Hence Taxer wins $R^+(H_w,h-h(w))$ since Taxer wins $R^+(H_w,f-f(w))$ by hypothesis.  
\end{proof}
We now produce some corollaries of Propositions \ref{LMinor}, \ref{-Minor}, and \ref{+Minor}, framed by statements concerning the original ranking problem.  For notational convenience, we say that $G$ being $f$-list rankable is equivalent to Ranker having a winning strategy for the game $R^{\ell}(G,f)$.  Fix $*\in\{\ell,-,+,\pm\}$.  

Suppose $G'$ is a subgraph of $G$, and $f:V(G)\rightarrow\mathbb{N}$ and $f':V(G')\rightarrow\mathbb{N}$ satisfy $f'(w)\geq f(w)$ for all $w\in V(G')$.  If $G$ is $f$-rankable, then $G'$ is clearly $f'$-rankable, since the restriction to $G'$ of an $f$-ranking of $G$ is an $f'$-ranking of $G'$.  A more general statement applies to the list versions of ranking, though not to the original ranking problem, as we shall see.  
\begin{cor}\label{Monotone}
Let $G'$ be a minor of a graph $G$, and suppose $f:V(G)\rightarrow\mathbb{N}$ and $f':V(G')\rightarrow\mathbb{N}$ satisfy $f'(w)\geq\min_{u\in U(w)}f(u)$ for all $w\in V(G')$.  If Ranker wins $R^*(G,f)$, then Ranker also wins $R^*(G',f')$.
\end{cor}
\begin{proof}
The statement follows from applying either Proposition \ref{LMinor}, \ref{-Minor}, or \ref{+Minor}, since clearly Taxer wins $R^*(H_w,f-f'(w))$ for each $w\in V(G')$ such that $f(u)-f'(w)\leq 0$ for some $u\in U(w)$.
\end{proof}
To see why Corollary \ref{Monotone} cannot be extended to the original ranking problem, let $n\geq 3$ and consider $G=P_n$ with vertices $v_1,\ldots,v_n$ in order, and construct the minor $G'$ of $G$ by contracting edges of $G$ until only a single edge $xy$ remains.  Let $f(v_1)=f(v_n)=f'(x)=f'(y)=1$ and $f(v_i)=n$ for $1<i<n$.  We can give $G$ an $f$-ranking by labeling $v_1$ and $v_n$ with $1$ and $v_i$ with $i$ for $1<i<n$, but we cannot give $G'$ an $f'$-ranking because $x$ and $y$ would both have to be labeled with $1$ even though they are adjacent.

Clearly a graph is $f$-rankable if and only if each of its components is $f$-rankable, and this statement also holds for the list versions of ranking.
\begin{cor}\label{Component}
Ranker wins $R^*(G,f)$ if and only if Ranker wins $R^*(G',f)$ for each component $G'$ of $G$.
\end{cor}
\begin{proof}
If Ranker wins $R^*(G,f)$, then Ranker wins $R^*(G',f)$ for each component $G'$ of $G$, by Corollary \ref{Monotone}.  Now suppose Ranker wins $R^*(G',f)$ for each component $G'$ of $G$.  If $*=\ell$, then Ranker can win $R^*(G,f)$ since each component can be dealt with individually.  If $*\in\{-,+,\pm\}$, then Ranker can win $R^*(G,f)$ by treating each move by Taxer on $G$ as a collection of separate moves on the games $R^*(G',f)$ and playing winning strategies for each of those games.
\end{proof}
We shall see that the following statement does not hold for the original ranking problem.
\begin{cor}\label{EdgeContraction}
Suppose $G'$ is obtained from $G$ by contracting an edge $uv$ into a vertex $w$, and $f(u)=f(v)=f(w)+1$.  If Ranker wins $R^*(G,f)$, then Ranker wins $R^*(G',f)$.
\end{cor}
\begin{proof}
Apply either Proposition \ref{LMinor}, \ref{-Minor}, or \ref{+Minor}, since clearly Taxer wins $R^*(H_w,f-f(w))$ if $H_w$ consists of an edge $uv$ such that $f(u)-f(w)=f(v)-f(w)=1$.
\end{proof}
To see why Corollary \ref{EdgeContraction} cannot be extended to the original ranking problem, let $n\geq 6$ and consider $G=P_n$ with vertices $v_1,\ldots,v_n$ in order and the minor $G'$ of $G$ obtained by contracting the edge $v_{n-1}v_{n-2}$ into the vertex $w$.  Let $f(v_i)=n$ for $1\leq i\leq n-3$ while $f(v_{n-2})=f(w)=1$ and $f(v_{n-1})=f(v_n)=2$.  We can give $G$ an $f$-ranking by labeling $v_i$ with $i$ for $1\leq i\leq n-3$, $v_{n-2}$ and $v_n$ with $1$, and $v_{n-1}$ with $2$, but we cannot give $G'$ an $f$-ranking because $v_{n-2}$ and $w$ would both have to be labeled with $1$ even though they are adjacent.
\section{Paths} 
\label{sec:paths}
To prove our first main result, that $\rho ^{\pm}_{\ell}(P_n)=\left\lceil \log (n+1)\right\rceil = \rho(P_n)$, we instead prove the stronger statement that $P_n$ is on-line $f$-list rankable if $\sigma _f(V(P_n))<1$, where $\sigma_f(V)=\sum_{v\in V}2^{-f(v)}$ for a nonnegative integer-valued function $f$ defined on a set $V$ of vertices.

Throughout this section, we will refer to the vertices of $P_n$ be $v_1,\ldots ,v_n$ from left to right.  Recall that $R^{\pm}(P_n,f)$ starts with each $v_i$ having $f(v_i)$ tokens, with Taxer beginning play by declaring the round to be low or high and then taking one token from each element of a nonempty set $T$ of vertices of $P_n$.  If the round is low, then Ranker responds by choosing an independent set $R\subseteq T$ to remove from $P_n$, replacing each removed vertex with an edge between its neighbors to get a path on $n-|R|$ vertices.  If the round is high, then Ranker responds by choosing a vertex $v\in T$ to delete from $P_n$ to get two path components on a total of $n-1$ vertices.

We isolate the following argument as a lemma because it alone is enough to show that $P_n$ is on-line $f$-list low-rankable (and thus $f$-list rankable) if $\sigma _f(V(P_n))<1$.
\begin{lem}\label{UpPath}
Suppose $\sigma _f(V(P_n))<1$, and Ranker has a winning strategy on $R^{\pm}(P_m,g)$ if $m<n$ and $\sigma _g(V(P_m))<1$.  If Taxer declares the first round of $R^{\pm}(P_n,f)$ low, then Ranker can win.
\end{lem}
\begin{proof}
Let $B=\{v_{2i-1}:1\leq i\leq\left\lceil n/2\right\rceil\}$ and $C=\{v_{2i}:1\leq i\leq\left\lfloor n/2\right\rfloor\}$.  Let $R=B\cap T$ if $\sigma_f(B\cap T)\geq\sigma_f(C\cap T)$ and $R=C\cap T$ otherwise.  Then $R$ is independent and $\sigma_f(R)\geq\sigma_f(T-R)$.  The game is reduced to $R^{\pm}(P_{n-|R|},g)$, where $g=f-1$ on $T-R$ and $g=f$ elsewhere.  Since $\sigma_f(R)\geq\sigma_f(T-R)$, we have $\sigma_g(V(P_n)-R)=\sigma_f(V(P_n))+\sigma_f(T-R)-\sigma_f(R)\leq \sigma_f(V(P_n))<1$.  By hypothesis Ranker wins this game and thus the original.
\end{proof} 
\begin{thm}\label{Path}
Ranker wins $R^{\pm}(P_n,f)$ if $\sigma _f(V(P_n))<1$.
\end{thm}
\begin{proof}
We use induction on $n$.  Clearly Ranker wins $R^{\pm}(P_1,f)$ when $f\geq 1$, so we assume $n>1$ and that Ranker has a winning strategy on $R^{\pm}(P_m,g)$ when $m<n$ and $\sigma_g(V(P_m))<1$.  By Lemma \ref{UpPath} we may also assume Taxer declares the first round high, so Ranker must remove one vertex $v$ from the set $T$ of vertices from which Taxer removes a token.  We will visit the vertices of $P_n$ in some order $v_{p_1},\ldots ,v_{p_n}$ such that $\{v_{p_1},\ldots ,v_{p_i}\}$ induces a path (not necessarily with the vertices in that order) for $1\leq i\leq n$, and if $i$ is the least index such that $v_{p_i}\in T$, then letting $v=v_{p_i}$ leads to a winning strategy for Ranker.  We will let $V^{<i}=\left\{ v_1,\ldots ,v_{p_i-1}\right\}$ and $P^{<i}$ be the subgraph of $P_n$ induced by $V^{<i}$, and we will let $V^{>i}=\left\{ v_{p_i+1},\ldots ,v_n\right\}$ and $P^{>i}$ be the subgraph of $P_n$ induced by $V^{>i}$.  Let $g=f-1$ on $T$ and $g=f$ elsewhere, and for $1\leq i\leq n$ let $g_i=f$ on $v_{p_1},\ldots ,v_{p_i}$ and $g_i=f-1$ elsewhere.  

For $1\leq i\leq n$, if $i$ is the least index such that $v_{p_i}\in T$, and if furthermore $\sigma_{g_i}(V^{<i})<1$ and $\sigma_{g_i}(V^{>i})<1$, then we claim that setting $v=v_{p_i}$ leads to a winning strategy for Ranker.  Indeed, setting $v=v_{p_i}$ reduces the game to separate games of $R^{\pm}(P^{<i},g)$ and $R^{\pm}(P^{>i},g)$, both of which Ranker wins by our inductive hypothesis: $P^{<i}$ and $P^{>i}$ each have fewer than $n$ vertices, and $g\geq g_i$ since $v_{p_j}\notin T$ for $1\leq j\leq i-1$, yielding $\sigma_g(V^{<i})\leq\sigma_{g_i}(V^{<i})<1$ and $\sigma_g(V^{>i})\leq\sigma_{g_i}(V^{>i})<1$.

We construct our ordering $v_{p_1},\ldots ,v_{p_n}$ of $V(P_n)$ inductively.  Select $p_1$ as the least index such that $\sigma_f(V^{<1}\cup\{v_{p_1}\})$ reaches at least $1/2$.  Since $g_1\geq f-1$ we have $\sigma_{g_1}(V^{<1})\leq 2\sigma_f(V^{<1})<1$ and $\sigma_{g_1}(V^{>1})\leq 2\sigma_f(V^{<1})<1$.

Now assume that $p_1,\ldots ,p_k$ have been found such that $\left\{v_{p_1},\ldots ,v_{p_k} \right\}$ induces a path and such that $\sigma_{g_i}(V^{<k})<1$ and $\sigma_{g_i}(V^{>k})<1$ for $1\leq i\leq k$.  Let $P_n-\left\{ v_{p_1},\ldots ,v_{p_k}\right\}$ consist of the paths induced by $\left\{ v_1,\ldots ,v_s\right\}$ and $\left\{ v_t,\ldots ,v_n\right\}$, where we set $s=0$ or $t=n+1$ if $v_1$ or $v_n$ is in $\left\{ v_{p_1},\ldots ,v_{p_k}\right\}$, respectively.  If $s\geq 1$ and we let $p_{k+1}=s$ then $\sigma_g(V^{<k+1})<\sigma_g(V^{<k})<1$, and if $t\leq n$ and we let $p_{k+1}=t$ then $\sigma_g(V^{>k+1})<\sigma_g(V^{>k})<1$.  

As $\sum ^{s}_{i=1}2^{1-f(v_i)}+2\sum ^{t-1}_{i=s+1}2^{-f(v_i)}+\sum ^{n}_{i=t}2^{1-f(v_i)}=2\sigma_f(V(P_n)) < 2$, at least one of the following holds: $s\geq 1$ and $\sum ^{t-1}_{i=s+1}2^{-f(v_i)}+\sum ^{n}_{i=t}2^{1-f(v_i)}< 1$, or $t\leq n$ and $\sum ^{s}_{i=1}2^{1-f(v_i)}+\sum ^{t-1}_{i=s+1}2^{-f(v_i)}< 1$.  If the former holds then letting $p_{k+1}=s$ results in $\sigma_{g_{k+1}}(V^{>k+1})< 1$, and if the latter holds then letting $p_{k+1}=t$ results in $\sigma_{g_{k+1}}(V^{<k+1})< 1$.  Thus Ranker can win.
\end{proof}
Naturally we want to see what happens when $\sigma_f{V(P_n)}\geq 1$.  To show the sharpness of Theorem \ref{Path}, we exhibit a function $f$ such that $\sigma_f(V(P_n))=1$ but $P_n$ is not even $f$-rankable, much less on-line $f$-list rankable.
\begin{prop}\label{PathSharp}
For $n\geq 1$, there is a function $f$ such that $\sigma_f(V(P_n))=1$ and $P_n$ is not $f$-rankable.
\end{prop}
\begin{proof}
Fix $n$, and set $k=\left\lfloor\log n\right\rfloor$.  Define $f(v_i)=k$ if $n-2^k<i\leq 2^k$ and $f(v_i)=k+1$ otherwise.  If $n=2^k$, then $P_n$ is not $f$-rankable because $\rho(P_n)=k+1$.  If $n>2^k$, then only one vertex $v_i$ can be labeled $k+1$, and $i$ must satisfy $1\leq i\leq n-2^k$ or $2^k<i\leq n$.  Thus one component of $P_n-v_i$ is a path on at least $2^k$ vertices which must be given a $k$-ranking, which is impossible.
\end{proof}
From Proposition \ref{PathSharp} it may seem hopeful that a perfect converse of Theorem \ref{Path} holds, but unfortunately one does not.
\begin{prop}\label{PathCounterex}
For $n\geq 4$, there is a function $f$ such that $\sigma_f(V(P_n))=1$ and Ranker wins $R^{\pm}(P_n,f)$.
\end{prop}
\begin{proof}
Let $f(v_1)=2$, $f(v_2)=3$, $f(v_3)=1$, $f(v_i)=i$ for $4\leq i\leq n-1$, and $f(v_n)=n-1$.  Then $\sigma_f(V(P_n))=1$.  Suppose the first round is low.  If $n\geq 5$ and Taxer removes tokens from $v_{n-1}$ and $v_n$, then Ranker can remove $v_n$ to reduce the game to $R^{\pm}(P_{n-1},f')$, where $f'(v_{n-1})=n-2$ and $f'=f$ elsewhere, so we may assume $n=4$ or $T\neq\{v_{n-1},v_n\}$.  If Ranker removes an independent set $R\subseteq T$ that maximizes $\sigma_f(R)$ then the game will reduce to $R^{\pm}(P_{n-|R|},g)$ where $\sigma_g(V(P_{n-|R|}))<1$, so Ranker wins by Theorem \ref{Path}.  

Now suppose the first round is high.  We continue use of the method and notation from the proof of Theorem \ref{Path}, visiting the vertices of $P_n$ in some sequence $v_{p_1},\ldots ,v_{p_n}$.  Let $p_1=3$, $p_2=2$, $p_3=1$, and $p_i=i$ for $4\leq i\leq n$.  For $1\leq i\leq n$, if $i$ is the least index such that $v_{p_i}\in T$, then $\sigma_{g_i}(V^{<i})<1$ and $\sigma_{g_i}(V^{>i})<1$.  By Theorem \ref{Path}, Ranker can win $R^{\pm}(P_n,f)$ by removing $v_{p_i}$ in the first round and then winning $R^{\pm}(P^{<i},g_i)$ and $R^{\pm}(P^{>i},g_i)$.
\end{proof}
\begin{cor}\label{Sigma}
For $n\geq 5$, there is a function $f$ such that $\sigma_f(V(P_n))>1$ and Ranker wins $R^{\pm}(P_n,f)$.
\end{cor}
\begin{proof}
By the Proposition \ref{PathCounterex}, if $f(v_1)=2$, $f(v_2)=3$, $f(v_3)=1$, and $f(v_4)=3$, then Ranker wins $R^{\pm}(P_4,f)$.  Let $f(v_i)=i$ for $5\leq i\leq n$, so $\sigma_f(V(P_n))=17/16-2^{-n}$.  By Lemma \ref{Distinct}, Ranker wins $R^{\pm}(P_n,f)$.
\end{proof}
The natural remaining question concerns what happens when $\sigma_f(V(P_n))$ is large.
\begin{con}
There exists a real number $\alpha$ such that for any positive integer $n$, if $\sigma_f(V(P_n))>\alpha$, then Taxer wins $R^{\pm}(P_n,f)$. 
\end{con}
\section{Cycles} 
\label{sec:cycles}
We now turn our attention to cycles, for which the following results are proved using many of the techniques and results of Section \ref{sec:paths}.  To prove our second main result, that $\rho ^{\pm}_{\ell}(C_n)=1+\left\lceil \log n\right\rceil = \rho(C_n)$, we once again prove a stronger statement, though we begin with a technical lemma.  

For a graph $G$, suppose $X=\{x_1,\ldots,x_n\}\subseteq V(G)$, where $f:X\rightarrow\mathbb{N}$ and $f(x_1)\leq\cdots\leq f(x_n)$.  Set $k=n$ if $f(x_1)<\cdots <f(x_n)$ and otherwise let $k$ satisfy $f(x_1)<\cdots <f(x_k)=f(x_{k+1})$.  For $X'=\{x_i\in X:1\leq i<k\}$ and $X''=\{x_i\in X:k<i\leq n\}$, define $\tau_f(X)=\sigma_f(X')+2\sigma_f(X'')$ (recalling that $\sigma_f(V)=\sum_{v\in V}2^{-f(v)}$).
\begin{lem}\label{CycleDown}
If $\tau_f(V(C_n))<1$, then Ranker wins $R^{\pm}(C_n,f)$ whenever Taxer declares the first round high.
\end{lem}
\begin{proof}
Recall that $R^{\pm}(C_n,f)$ starts with each $v\in V(C_n)$ having $f(v)$ tokens.  If Taxer begins play by declaring the first round high, then Taxer takes one token from each element of a nonempty set $T\subseteq V(C_n)$, and Ranker responds by choosing some $u\in T$ to delete from $C_n$ to get a path on $n-1$ vertices.  Set $Y=V(C_n)-T$ and $Z=T-\{u\}$, and define $\tau '=\sigma_f(Y)+2\sigma_f(Z)$.  Thus after the first round the game is reduced to $R^{\pm}(P_{n-1},g)$, where $V(P_{n-1})=V(C_n)-\{u\}$ and $g(v)=f(v)$ for $v\in Y$ and $g(v)=f(v)-1$ for $v\in Z$.  Note that $\sigma_g(V(P_{n-1}))=\tau'$, so by Theorem \ref{Path} Ranker wins $R^{\pm}(C_n,f)$ if Ranker can always choose some $u\in T$ so that $\tau'<1$.  

Letting $V(C_n)=\{v_1\ldots,v_n\}$, named so that $f(v_1)\leq\cdots\leq f(v_n)$, we note that if $f(v_i)<f(v_{i+1})$ for $1\leq i<n$, then Ranker wins $R^{\pm}(C_n,f)$ by Lemma \ref{Distinct}.  Thus we may assume $f$ is not injective, and let $k$ be the least index such that $f(v_k)=f(v_{k+1})$.  We complete the proof by showing that the minimum value Ranker can make $\tau'$ (given Taxer's choice of $T$) is maximized when $T=\{v_k,\ldots,v_n\}$.  Indeed, for that choice of $T$, Ranker can set $u=v_k$ to get $\tau '=\tau_f(V(C_n))<1$, so for any other $T$ Ranker could choose some $u\in T$ so that $\tau'<1$.  For the rest of the proof, assume Taxer has chosen $T$ to maximize the minimum value Ranker is able to make $\tau'$ through the selection of $u$.

For any choice of $T$, Ranker minimizes $\tau '$ by selecting $u$ as the vertex in $T$ having the fewest tokens.  Thus we may assume Taxer selects $T=\{v_j,\ldots ,v_n\}$ for some $j$ satisfying either $j=1$ or $f(v_{j-1})<f(v_j)$, since elements of $Z$ get counted double in $\tau '$.  Given this choice of $T$, Ranker will select $u=v_j$ to get $\tau '=\sigma_f(Y)+2\sigma_f(Z)$ for $Y=\{v_1,\ldots,v_{j-1}\}$ and $Z=\{v_{j+1},\ldots ,v_n\}$.  We show $\tau'$ is minimized when $j=k$.

We first show that $j\leq k$ by showing that if $f(v_i)=f(v_{i+1})<f(v_j)$, then Taxer could increase $\tau '$ by $2^{1-f(v_j)}$ by adding $v_i$ and $v_{i+1}$ to $T$.  Indeed, Ranker would choose $u=v_i$, with $\tau '$ losing $2^{1-f(v_i)}$ by removing $v_i$ and $v_{i+1}$ from $Y$ but gaining $2^{1-f(v_i)}+2^{1-f(v_j)}$ by adding $v_{i+1}$ and $v_j$ to $Z$.

We now show that if $j<k$, then Taxer would not decrease $\tau '$ by removing $v_j,\ldots,v_{k-1}$ from $T$.  Indeed, $\tau '$ would gain $2^{-f(v_j)}$ by adding $v_j$ to $Y$ and only lose $2^{1-f(v_k)}+\sum ^{k-1}_{d=j+1}2^{-f(v_d)}$ by removing $v_k$ from $Z$ and switching $v_{j+1},\ldots ,v_{k-1}$ from $Z$ to $Y$, and we have $2^{1-f(v_k)}+\sum ^{k-1}_{d=j+1}2^{-f(v_d)}\leq 2^{1-f(v_j)-k+j}+\sum ^{k-j-1}_{d=1}2^{-f(v_j)-d}=2^{-f(v_j)}(2^{-k+j}+\sum ^{k-j}_{d=1}2^{-d})=2^{-f(v_j)}$.
\end{proof}
\begin{cor}\label{CycleDownCor}
If $\sigma_f(V(C_n))<1/2+2^{-\left\lceil \log n\right\rceil}$, then Ranker wins $R^{\pm}(C_n,f)$ whenever Taxer declares the first round high.
\end{cor}
\begin{proof}
By Lemma \ref{CycleDown} it suffices to show $\tau_f(V(C_n))<1$.  Assume without loss of generality that $f(v_1)<\cdots<f(v_k)=f(v_{k+1})\leq\cdots\leq f(v_n)$.  Set $Y=\{v_1,\ldots,v_{k-1}\}$ and $Z=\{v_{k+1},\ldots,v_n\}$.  If $f(v_k)\leq\left\lceil \log n\right\rceil$, then $\tau_f(V(C_n))\leq 2\sigma_f(V(C_n))-2^{1-f(v_k)}<1+2^{1-\left\lceil\log n\right\rceil}-2^{1-\left\lceil \log n\right\rceil}=1$.  If $f(v_k)\geq 1+\left\lceil \log n\right\rceil$, then we have $\tau_f(V(C_n))=\sigma_f(Y)+2\sigma_f(Z)=\sigma_f(V(C_n))-2^{-f(v_k)}+\sigma_f(Z)<1/2+2^{-\left\lceil \log n\right\rceil}-2^{-f(v_k)}+(n-1)2^{-f(v_k)}=1/2+2^{-\left\lceil \log n\right\rceil}+(n-2)2^{-f(v_k)}\leq 1/2+2^{-\left\lceil \log n\right\rceil}+(n-2)2^{-1-\left\lceil \log n\right\rceil}=1/2+n2^{-1-\left\lceil \log n\right\rceil}\leq 1$.
\end{proof}
\begin{thm}\label{Cycle}
If $\sigma_f(V(C_n))<1/2+2^{-\left\lceil \log n\right\rceil}$, then Ranker wins $R^{\pm}(C_n,f)$.
\end{thm}
\begin{proof}
By Corollary \ref{CycleDownCor} it suffices to show Ranker wins whenever Taxer declares the first round low, which we do using induction on $n$.  Let $V(C_n)=\{v_1,\ldots,v_n\}$ and $E(C_n)=\{v_iv_{i+1}:1\leq i\leq n\}$.  Note that $C_3$ is on-line $f$-list rankable when $\sigma_f(V(C_3))< 3/4$: if $f(v_1)\leq f(v_2)\leq f(v_3)$ and $\sigma_f(V(C_3))< 3/4$, then $f(v_i)\geq i$ for each $i$, so by Lemma \ref{Distinct} Ranker can win $R^{\pm}(C_3,f)$.  Now suppose Ranker wins $R^{\pm}(C_m,g)$ for $3\leq m<n$ and $\sigma_g(V(C_m))<1/2+2^{-\left\lceil \log m\right\rceil}$.

Recall that $R^{\pm}(C_n,f)$ starts with each $v_i$ having $f(v_i)$ tokens.  If Taxer begins play by declaring the first round low, then Taxer takes one token from each element of a nonempty set $T\subseteq V(C_n)$, and Ranker responds by choosing an independent set $R\subseteq T$ to remove from $C_n$, replacing each removed vertex with an edge between its neighbors to get a cycle on $n-|R|$ vertices (or an edge if $|R|=n-2$).  Let $B=\{v_{2i-1}:1\leq i\leq\left\lceil n/2\right\rceil\}$ and $C=\{v_{2i}:1\leq i\leq\left\lfloor n/2\right\rfloor\}$.

First suppose $T=V(C_n)$ and $n$ is odd, and without loss of generality assume $f(v_n)\geq f(v_i)$ for $1\leq i\leq n$.  Hence $f(v_n)\geq 1+\left\lceil\log n\right\rceil$, since otherwise we would have $\sigma_f(T)\geq (n-1)2^{-\left\lceil\log n\right\rceil}+2^{-\left\lceil\log n\right\rceil}\geq 1/2+2^{-\left\lceil \log n\right\rceil}$.  Set $R=B-\{v_n\}$ if $\sigma_f(B-\{v_n\})\geq\sigma_f(C)$ and $R=C$ otherwise, so $R$ is independent and the game reduces to $R^{\pm}(C_{(n+1)/2},f-1)$, where $V(C_{(n+1)/2})=T-R$ and $2\sigma_f(R)\geq\sigma_f(T)-2^{-f(v_n)}$.  By our inductive hypothesis Ranker wins, due to the  fact that $\sigma_{f-1}(T-R)=2\sigma_f(T)-2\sigma_f(R)\leq \sigma_f(T)+2^{-f(v_n)}< 1/2+2^{-\left\lceil \log n\right\rceil}+2^{-1-\left\lceil \log n\right\rceil}<1/2+2^{1-\left\lceil \log n\right\rceil}= 1/2+2^{-\left\lceil \log ((n+1)/2)\right\rceil}$.

Now suppose that $T\neq V(C_n)$ or $n$ is even, and replace the assumption that $f(v_n)\geq f(v_i)$ for $1\leq i\leq n$ with the assumption that $v_n\notin T$ if $n$ is odd.  Set $R=B\cap T$ if $\sigma_f(B\cap T)\geq\sigma_f(C\cap T)$ and $R=C\cap T$ otherwise, so $R$ is independent and $\sigma_f(R)\geq\sigma_f(T-R)$.  The game reduces to $R^{\pm}(C_{n-|R|},g)$, where $V(C_{n-|R|})=V(C_n)-R$ and $g(v)=f(v)-|T\cap\{v\}|$.  By our inductive hypothesis Ranker wins, as $\sigma_g(V(C_n)-R)=\sigma_f(V(C_n))+\sigma_f(T-R)-\sigma_f(R)\leq \sigma_f(V(C_n))<1/2+2^{-\left\lceil \log n\right\rceil}<1/2+2^{-\left\lceil \log (n-|R|)\right\rceil}$.
\end{proof}
Once again we want to explore the boundary case.
\begin{prop}
For $n\geq 3$, there is a function $f$ such that $\tau_f(V(C_n))=1$ and $\sigma_f(V(C_n))=1/2+2^{-\left\lceil \log n\right\rceil}$ but $C_n$ is not $f$-list rankable.
\end{prop}
\begin{proof}
Note that if $f(v)=k+1$ for each $v\in V(C_{2^k+1})$, then $\tau_f(V(C_{2^k+1}))=1$ and $\sigma_f(V(C_{2^k+1}))=1/2+2^{-\left\lceil \log (2^k+1)\right\rceil}$, but $C_{2^k+1}$ is not even $f$-rankable.  By Corollary \ref{EdgeContraction}, if $C_n$ is not $f$-list rankable and $C_{n+1}$ has the same vertices as $C_n$ except for replacing some vertex $w$ maximizing $f$ on $V(C_n)$ with adjacent vertices $u$ and $v$ in $C_{n+1}$, then setting $f(u)=f(v)=f(w)+1$ precludes $C_{n+1}$ from being $f$-list rankable.  By this construction $\tau_f(V(C_n))=\tau_f(V(C_{n+1}))$ and $\sigma_f(V(C_n))=\sigma_f(V(C_{n+1}))$, so the proposition follows.
\end{proof}
\begin{prop}
For $n\geq 6$, there is a function $f$ such that $\tau_f(V(C_n))=1$ and $C_n$ is $f$-list rankable.
\end{prop}
\begin{proof}
Let $V(C_n)=\{v_1,\ldots ,v_n\}$ and $E(C_n)=\{v_iv_{i+1}:1\leq i\leq n\}$, with $f(v_1)=1$, $f(v_2)=3$, $f(v_3)=4$, $f(v_4)=2$, $f(v_i)=i$ for $5\leq i<n$, and $f(v_n)=n-1$.  Let $L$ be an $f$-list assignment, and over $\bigcup_{i=1}^n L(v_i)$ let $m$ be the largest value, $m'$ be the second largest value, $b$ be the smallest value, and $b'$ be the second smallest value found.  We wish to find a ranking of $C_n$ such that each $v_i$ is labeled with $a_i\in L(v_i)$.
\begin{case}\label{one}
$L(v_1)\cap\{m,m'\}\neq\emptyset$.
\end{case}
We can create an $L$-ranking of $C_n$ by letting $a_1\in\{m,m'\}$ and then ranking the path $P$ induced by $\{v_2,\ldots,v_n\}$ so that $a_i\in L(v_i)-\{a_1\}$ for $2\leq i\leq n$.  Indeed, it was shown in  Proposition \ref{PathCounterex} that $P$ is $(f-1)$-list rankable, and our ranking of $P$ labels no vertex with $a_1$ and at most one vertex with a label greater than $a_1$.
\begin{case}\label{two}
$L(v_n)\neq\bigcup ^{n-1}_{i=1} L(v_i)$.
\end{case}
Let $j$ be the least index such that $L(v_j)-L(v_n)\neq\emptyset$.  We create an $L$-ranking by choosing labels $a_i\in L(v_i)$ in the order $i=1,4,2,3,5,\ldots ,n$ such that each $a_i$ is distinct from its predecessors and $a_j\notin L(V_n)$.  We can do this for $1\leq i<n$ since by the time $a_i$ is to be chosen only $f(v_i)-1$ previous labels will have been used, and we can choose $a_n\in L(v_n)-\{a_1,\ldots,a_{n-1}\}$ since $|L(v_n)|=n-1$ and $a_j\notin L(v_n)$. 
\begin{case}\label{three}
$b\in \bigcup ^{n-2}_{i=1} L(v_i)$.
\end{case}
We can choose $a_1,\ldots ,a_{n-2}$ from $L(v_1),\ldots ,L(v_{n-2})$ to be distinct and contain $b$, leaving some $a\in L(v_{n-1})-\left\{ a_1,\ldots ,a_{n-2}\right\}$ with $a>b$.  We may assume $a\in L(v_n)$ (otherwise Case \ref{two} applies), so we can complete an $L$-ranking by either setting $a_{n-1}=a$ and $a_n=b$ if $a_1\neq b$ or setting $a_{n-1}=b$ and $a_n=a$ if $a_{n-2}\neq b$. 
\begin{case}\label{four}
$L(v_1)=\left\{ b'\right\}$.
\end{case}
We may assume $b\notin L(v_2)$ (otherwise Case 3 applies), so finding an $L$-ranking of $C_n$ reduces to finding an $L'$-ranking of the cycle $C_{n-1}$ created by deleting $v_1$ and adding the edge $v_n v_2$, where $L'(v_i)=L(v_i)-\{b'\}$ for $i\in\{2,n-1,n\}$ and $L'(v_i)=L(v_i)$ for $3\leq i\leq n-2$.  Indeed, we would have $a_2>b'$ (since $L'(v_2)\cap\{b,b'\}=\emptyset$) and $\max\{a_{n-1},a_n\}>b'$ (since $b'\notin L'(v_{n-1})\cup L'(v_n)$ and $a_{n-1}\neq a_n$), so the $L'$-ranking of $C_{n-1}$ could be turned into an $L$-ranking of $C_n$ by setting $a_1=b'$.  If $f'(v_i)=f(v_i)-1$ for $i\in\{2,n-1,n\}$ and $f'(v_i)=f(v_i)$ for $3\leq i\leq n-2$, then $|L'(v_i)|\geq f'(v_i)$ for each $i$, and we have $\tau_{f'}(V(C_{n-1}))=2(1/4+(\sum ^{n-2}_{i=4}2^{-i})+2^{2-n}+2^{2-n}) = 2(3/8+2^{2-n})\leq 2(3/8+1/16)=7/8<1$.  By Lemma \ref{CycleDown}, $C_{n-1}$ is $f'$-list rankable, giving us an $L'$-ranking of $C_{n-1}$ and thus an $L$-ranking of $C_n$.  
\begin{case}\label{five}
$n=6$.
\end{case}
Without loss of generality assume $L(v_5)=L(v_6)=\{1,2,3,4,5\}$ (since Case \ref{two} applies if $L(v_6)\neq\bigcup ^{5}_{i=1} L(v_i)$), so $b=1$, $b'=2$, $m'=4$, and $m=5$.  Thus we may assume, lest Case \ref{three} apply, that $L(v_1)=\{ 3\}$ (otherwise Case \ref{one} or \ref{four} applies), $L(v_2)$ contains $2$ or $3$ as well as $4$ or $5$, $L(v_3)=\{ 2,3,4,5\}$, and either $L(v_4)=\{ 4,5\}$ or $L(v_4)$ contains $2$ or $3$.  If $2\in L(v_4)$, let $a_1=3$, $a_2=4$ (or $a_2=5$, if $4\notin L(v_2)$), $a_3=3$, $a_4=2$, $a_5=5$ (or $a_5=4$, if $4\notin L(v_2)$), and $a_6=1$.  If $3\in L(v_4)$, let $a_1=3$, $a_2=4$ (or $a_2=5$ if $4\notin L(v_2)$), $a_3=2$, $a_4=3$, $a_5=5$ (or $a_5=4$, if $4\notin L(v_2)$), and $a_6=1$.  If $L(v_4)=\{ 4,5\}$, let $a_1=3$, $a_2=4$ (or $a_2=5$, if $4\notin L(v_2)$), $a_3=2$, $a_4=5$ (or $a_4=4$, if $4\notin L(v_2)$), $a_5=2$, and $a_6=1$.
\begin{case}\label{six}
$n\geq 7$.
\end{case}
We perform induction on $n$, using Case \ref{six} as the base case.  Thus we assume the cycle $C_{n-1}$ created by deleting $v_{n-2}$ and adding the edge $v_{n-3}v_{n-1}$ is $f'$-list rankable, where $f'(v_i)=f(v_i)-1$ for $i\in\{n-1,n\}$ and $f'=f$ elsewhere.  We may also assume $b\in (L(v_{n-1})\cap L(v_n))-(L(v_1)\cup L(v_{n-2}))$ and $b'\in (L(v_{n-1})\cap L(v_n))-L(v_1)$ (otherwise Case \ref{three} or \ref{four} applies).  If $a_1,\ldots ,a_{n-2}$ can be chosen from $L(v_1),\ldots ,L(v_{n-2})$ to be distinct such that $a_{n-2}\neq b'$, then we can complete an $L$-ranking of $C_n$ by setting $a_{n-1}=b'$ and $a_n=b$.  If $a_1,\ldots ,a_{n-2}$ cannot be chosen from $L(v_1),\ldots ,L(v_{n-2})$ to be distinct such that $a_{n-2}\neq b'$, then the only lists containing $b'$ are $L(v_{n-2})$, $L(v_{n-1})$, and $L(v_n)$.  By our inductive hypothesis, if $L'(v_i)=L(v_i)$ for $1\leq i\leq n-3$ and $L'(v_i)=L(v_i)-\{b'\}$ for $i\in\{n-1,n\}$, then the cycle $C_{n-1}$ created by deleting $v_{n-2}$ and adding the edge $v_{n-3}v_{n-1}$ has an $L'$-ranking, which we can extend to an $L$-ranking of $C_n$ by setting $a_{n-2}=b'$. 
\end{proof}
\setcounter{case}{0}
\begin{cor}
For $n\geq 7$, there is a function $f$ such that $\tau_f(V(C_n))>1$ and $C_n$ is $f$-list rankable.
\end{cor}
\begin{proof}
By the above proposition, if $f=(1,3,4,2,5,5)$ then $C_6$ is $f$-list rankable.  Let $f(v_i)=i$ for $7\leq i\leq n$; then $\tau_f(V(C_n))=33/32-2^{1-n}$ and $C_n$ is $f$-list rankable by Lemma \ref{Distinct}.
\end{proof}
Conjecture \ref{Sigma} says that if $\sigma_f(V(P_n))$ is large enough, then $P_n$ is not $f$-list rankable.  Since deleting an edge of $C_n$ leaves a copy of $P_n$ and $\tau_f(V(C_n))<2\sigma_f(V(C_n))$, Conjecture \ref{Sigma} would also imply that $C_n$ is not $f$-list rankable for large enough $\tau_f(V(C_n))$ or $\sigma_f(V(C_n))$.
\section{Trees With Many Leaves} 
\label{sec:trees}
We now prove our third main result, that $\rho _{\ell}(T)=q$ if $T$ is a tree having $p$ internal vertices and $q$ leaves, where $q\geq 2^{p+2}-2p-4$.  Since Proposition \ref{TreeBound} implies that $\rho _{\ell}(T)\geq q$ for any tree $T$ with $q$ leaves, we need only prove the upper bound.  We consider separately trees with two or fewer internal vertices.
\begin{prop}\label{Star}
If $S$ is a star with $q$ leaves, then $\rho ^+_{\ell}(S)=q$.
\end{prop}
\begin{proof}
We use induction on $q$ to show Ranker has a winning strategy for the game $R^+(S,f)$, where $f=q$ everywhere.  The statement is obvious if $S\in\{K_1,K_2\}$, and it follows from Theorem \ref{Path} if $S=P_3$.  Thus we may assume $q\geq 3$ and Ranker wins for stars having fewer than $q$ leaves. 

If Taxer takes a token from the internal vertex in the first round, let Ranker respond by removing it; then $q$ isolated vertices remain, each with at least $q-1$ tokens, so Ranker can win the game.  If Taxer takes tokens from only leaves in the first round, then let Ranker respond by removing a leaf, leaving a star with $q-1$ leaves and at least $q-1$ tokens on each vertex.  By the inductive hypothesis Ranker can win the game.
\end{proof}
\begin{prop}\label{DoubleStar}
If $T$ is a double star with $q$ leaves ($q\geq 3$), then $\rho ^+_{\ell}(T)=q$.
\end{prop}
\begin{proof}
We show Ranker has a winning strategy for the game $R^+(T,f)$, where $f=q$ everywhere.  Let the (adjacent) interior vertices of $T$ be $x$ and $y$, with $x$ adjacent to leaves $x_1,\ldots ,x_m$ and $y$ adjacent to leaves $y_1,\ldots ,y_n$ (so $q=m+n\geq 3$).  We may assume $m\leq n$, so by hypothesis $n\geq 2$.  If Taxer selects an internal vertex in the first round, let Ranker respond by removing an internal vertex, leaving behind isolated vertices and a star with at most $q-1$ leaves.  Each remaining vertex still has at least $q-1$ tokens, so by Proposition \ref{Star} Ranker can win the game.  

Thus we may assume that Taxer selects only leaves in the first round, with Ranker responding by removing a selected leaf.  We use induction on $q$ to show Ranker has a winning strategy.  If $q=3$, then $m=1$ and $n=2$.  If Ranker removes $y_1$ or $y_2$, then a path on four vertices remains, with the internal vertices having three tokens each and the leaves having at least two tokens each; by Theorem \ref{Path} Ranker can win the game.  If Taxer only selects $x_1$ and Ranker removes $x_1$, then a star with three leaves remains, with each vertex having three tokens; by Proposition \ref{Star} Ranker can win this game.  

Now assume $q\geq 4$ and Ranker wins for trees having two internal vertices and between three and $q-1$ leaves.  If Ranker removes some $y_i$, then a tree with $q-1$ leaves and two internal vertices remains, with each vertex having at least $q-1$ tokens; by the inductive hypothesis Ranker can win this game.  If Taxer only selects leaves adjacent to $x$, then Ranker will remove some $x_i$, leaving behind either a tree with two internal vertices and $q-1$ leaves, with each vertex having at least $q-1$ tokens, or a star with $q$ leaves, with each vertex having $q$ tokens.  Either way Ranker can win this game.
\end{proof}
\begin{thm}\label{Tree}
For any tree $T$ having $p$ internal vertices and $q$ leaves, if $q\geq 2^{p+2}-2p-4$, then $\rho _{\ell}(T)=q$.
\end{thm}
\begin{proof}
If $p<\min\{3,q\}$, then $\rho _{\ell}(T)\leq\rho ^+_{\ell}(T)=q$ by Propositions \ref{Star} and \ref{DoubleStar}, so we may assume $p\geq 3$.  If $T$ is a tree with $p$ internal vertices and $q$ leaves, with $q\geq 2^{p+2}-2p-4$, then $T$ has a vertex of degree at least $3$.  For an internal vertex $u$, if $u$ is a vertex of degree at least $3$, or is adjacent to one, or is located on a path whose endpoints each have degree at least $3$ in $T$, let $T_u$ be a component of $T-u$ containing the most leaves of $T$.  If $u$ is any other internal vertex, let $T_u=T_w$, where $w$ is the unique vertex nearest to $u$ that has degree $2$ and is adjacent to a vertex of degree at least $3$.

For each internal vertex $u$, say $T_u$ has $p_u$ internal vertices and $q_u$ leaves, $q'_u$ of which are also leaves of $T$.  Clearly $p_u<p$ and $q_u'\leq q_u<q$.  Let $v$ be an internal vertex such that $q'_v$ is minimum.  For any internal vertex $u$ besides $v$, we have $q'_u\geq q/2$ since either $q'_v\geq q/2$, in which case $q'_u\geq q/2$ by the minimality of $q'_v$, or $q'_v< q/2$, in which case any subtree of $T$ obtained by deleting from $T$ a component of $T-v$ must have more than $q/2$ leaves of $T$.  Thus $v$ has degree at least $3$, and $T_u$ contains the subtree of $T$ obtained by deleting the component of $T-v$ containing $u$, giving $T_u$ more than $q/2$ leaves of $T$.

Let $L$ be a $q$-uniform list assignment to $T$, and for any internal vertex $u$ let $m_u$ denote the largest element of $L(u)$.  Call an internal vertex $u$ \emph{special} if $q'_u\geq q/2$ (that is, if $u\neq v$ or $u=v$ and $q'_v\geq q/2$) and there are vertices $u_1,\ldots ,u_p$ that are leaves of both $T$ and $T_u$ and satisfy the following properties: each internal vertex in $T_u$ has at least two neighbors in $T_u$ that are not one of these leaves, and from each $L(u_i)$ we can select some $e_i$ such that $e_1<\cdots <e_p<m_u$.  We classify $L$ by whether $L$ admits a special vertex, and in each case we show how to give $T$ an $L$-ranking.
\begin{case}
$L$ admits no special vertex.
\end{case}
If $q'_v<q/2$ and $m_v$ is the largest label in any list, label $v$ with it; since no component of $T-v$ can have $q/2+p-1$ vertices, and $q/2+p-1\leq q-1$, a ranking can be completed by deleting $m_v$ from the list of any unlabeled vertex, and then for each component of $T-v$ giving distinct labels to each vertex.  Now assume $q'_v\geq q/2$ or $m_v$ is not the largest label in any list.  Give each leaf a separate label (which is possible since there are $q$ leaves and each vertex receives a list of size $q$), making sure to give some leaf a label larger than $m_v$ if possible.  

Since for each internal vertex $u$ besides $v$, $T_u$ contains fewer than $p$ internal vertices of $T$ and at least $q/2$ leaves of $T$, we can fix a set $S_u$ of at least $q/2-p+1$ leaves of both $T$ and $T_u$ such that each internal vertex in $T_u$ has at least two neighbors in $T_u$ that are not in $S_u$ (to get $S_u$, delete from the set of leaves in both $T$ and $T_u$ one leaf adjacent to each of the at most $p-1$ internal vertices of $T_u$ adjacent to a leaf of $T$ in $T_u$, and delete an additional leaf if $T_u$ is a star).  

For each internal vertex $u$ besides $v$, $L(u)$ contains at most $p$ of the labels used on $S_u$, since otherwise $u$ would be a special vertex (with $u_1,\ldots,u_p$ being the elements of $S_u$ receiving the smallest labels).  Then $L(u)$ contains at most $q-((q/2-p+1)-p)$, or $q/2+2p-1$, of the labels used on the leaves of $T$, so deleting from each $L(u)$ the labels used on the leaves of $T$ yields a list of size at least $q-(q/2+2p-1)$, or $q/2-2p+1$, which is greater than $p$ for $p\geq 3$.  If $q'_v\geq q/2$, then the same holds for $L(v)$, and we can complete a ranking by giving distinct labels to each of the $p$ internal vertices.  

If $q'_v<q/2$, then by hypothesis $m_v$ is not the largest label in any list.  In this case the largest label must be in the list of some leaf (or else the internal vertex $u$ containing that label would be a special vertex, with $u_1,\ldots,u_p$ being any $p$ elements of $S_u$), and we assigned that label to such a leaf.  Thus we can complete a ranking by labeling $v$ distinctly from the leaves (possible since $|L(v)|=q$ and one of the $q$ leaves was given a label not in $L(v)$) and then labeling the remaining internal vertices distinctly.
\begin{case}
$L$ admits a special vertex $u$.
\end{case}
We use induction on $p$; assume $p\geq 3$.  If $u$ has degree $2$ and is not adjacent to any vertex of degree at least $3$, and some component of $T-u$ is a path, then let $w$ be the vertex nearest to $u$ that has degree $2$ and is adjacent to a vertex of degree at least $3$, and without loss of generality assume no vertex between $u$ and $w$ is special.

Label $u$ with $m_u$ and each $u_i$ with $e_i$, so by the positioning of $u$ and the size of its label, no label given to a vertex separated from $T_u$ by $u$ can cause a conflict with the label of any $u_i$.  If we delete $m_u,e_1,\ldots ,e_p$ from the list of each of the $p_u+q_u-p$ unlabeled vertices in $T_u$ then each such list must still have size at least $q-p-1$, and thus we can finish ranking $T_u$ if the subtree induced by its unlabeled vertices is $(q-p-1)$-list rankable (since no vertex already labeled could be part of a path in $T_u$ between vertices with the same label, once the remaining vertices are labeled from their truncated lists).  By our inductive hypothesis this subtree is in fact $(q-p-1)$-list rankable, since it has $p_u$ internal vertices and $q_u-p$ leaves, with $p_u\leq p-1$ and $q_u-p\geq q/2-p\geq (2^{p+2}-2p-4)/2-p=2^{p+1}-2p-2=2^{(p-1)+2}-2(p-1)-4$.  Thus we can finish ranking $T_u$.

Let $A$ be the set of vertices separated from $T_u$ by $u$, with $b$ denoting the minimum size of a list assigned to a vertex in $A$ after deleting any labels no smaller than $m_u$ that were used on $u$ or $T_u$, and let $A'$ be the (possibly empty) set of vertices strictly between $u$ and $T_u$, with $b'$ denoting the minimum size of a list assigned to a vertex in $A'$ after deleting any labels used on $u$ or $T_u$.  We complete the proof by showing $|A|\leq b$ and $|A'|\leq b'-|A|$: then a ranking of $T$ can be completed trivially by giving vertices in $A$ distinct labels from their truncated lists (these labels cannot come into conflict with any labels given to $T_u$ since they are separated from each other by the label $m_u$ given to $u$) and then giving each of the vertices in $A'$ distinct and previously unused labels.  
\begin{figure}[ht]
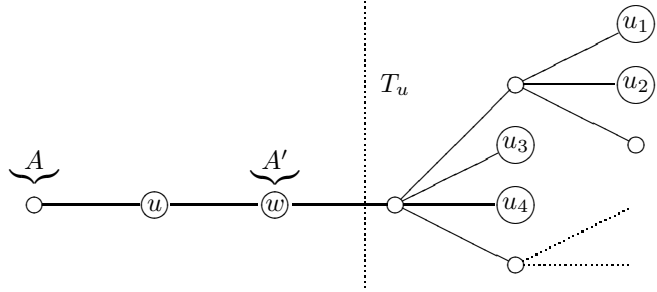

\centering
\[ \xygraph{
!{<0cm,0cm>;<1.6cm,0cm>:<0cm,.8cm>::}
!{(0,0) }*+[o][F-:<3pt>]{}="a"
!{(1,0) }*+[o][F-:<3pt>]{u}="b"
!{(2,0) }*+[o][F-:<3pt>]{w}="c"
!{(3,0) }*+[o][F-:<3pt>]{}="d"
!{(4,2) }*+[o][F-:<3pt>]{}="e"
!{(5,3) }*+[o][F-:<3pt>]{u_1}="f"
!{(5,2) }*+[o][F-:<3pt>]{u_2}="g"
!{(5,1) }*+[o][F-:<3pt>]{}="h"
!{(4,1) }*+[o][F-:<3pt>]{u_3}="i"
!{(4,0) }*+[o][F-:<3pt>]{u_4}="j"
!{(4,-1) }*+[o][F-:<3pt>]{}="k"
!{(2,.75) }*+{A'}*\frm{_\}}
!{(0,.75) }*+{A}*\frm{_\}}
!{(3,2) }*+{T_u}
!{(2.75,3.5) }*+{}="p"
!{(2.75,-1.5) }*+{}="q"
!{(5,0) }*+{}="r"
!{(5,-1) }*+{}="s"
"a"-"b"-"c"-"d"-"e"
"d"-"i" "d"-"j" "d"-"k"
"e"-"f" "e"-"g" "e"-"h"
"p"-@{.}"q"
"k"-@{.}"r"
"k"-@{.}"s"
} \]
\caption{A possibility for $T$, if $u$ is a special vertex and $A'\neq\emptyset$.}
\end{figure}

We have $|A|=p+q-p_u-q_u-|A'|-1$, since $T$ has $p+q$ vertices and $A$ does not include the $p_u+q_u$ vertices of $T_u$, nor the vertices of $A'$, nor $u$.  We also have $b\geq q-(p_u+q_u-p+1)$, since each list assigned to an unlabeled vertex started out with $q$ elements, and the only ones that could have been deleted were $m_u$ as well as any of the at most $p_u+q_u-p$ labels given to $T_u$ that exceeded $m_u$.  Thus $|A|\leq b-|A'|$.

The only vertices left to label are those in $A'$, if $A'$ is nonempty.  Since for each $z\in A'$, $T_z$ contains fewer than $p$ internal vertices of $T$ and exactly $q-1$ leaves of $T$, we can fix a set $S_z$ of at least $q-p$ leaves of both $T$ and $T_z$ such that each internal vertex in $T_z$ has at least two neighbors in $T_z$ that are not in $S_z$ (to get $S_z$, delete from the set of leaves in both $T$ and $T_z$ one leaf adjacent to each of the at most $p-1$ internal vertices of $T_z$ adjacent to a leaf of $T$ in $T_z$, and delete an additional leaf if $T_z$ is a star).  
  
For each $z\in A'$, $L(z)$ contains at most $p$ of the labels used on $S_z$, since otherwise $z$ would be a special vertex (with $z_1,\ldots,z_p$ being the elements of $S_z$ receiving the smallest labels).  Then $L(z)$ contains at most $2p-1$ of the labels used on the leaves of $T_z$, so deleting from each $L(z)$ the labels used on the leaves of $T_z$ or on any of the other $p$ vertices of $T$ besides $z$ (i.e., the leaf of $T$ not in $T_z$ along with any of the $p-1$ internal vertices of $T$ besides $z$) yields a list of size at least $q-3p+1$.  Thus we can finish ranking $T$ because $|A'|\leq p-2\leq q-3p+1\leq b'-|A|$ for $p\geq 3$.
\end{proof}
We conclude by noting that no statement similar to Theorem \ref{Tree} can be applied to graphs in general.
\begin{prop}
For $p\geq 3$ and $q\geq 0$, there is a graph $G$ with $p$ internal vertices and $q$ pendant vertices such that $\rho _{\ell}(G)>q$.
\end{prop}
\begin{proof}
Let $G$ be obtained by connecting $q$ pendant vertices to the complete graph $K_p$ such that at least one internal vertex $v$ is not adjacent to any pendant vertex.  Then $G$ contains a spanning subtree of which $v$ is a leaf; this tree has $q+1$ leaves, so $\rho _{\ell}(G)>q$ by Proposition \ref{TreeBound}. 
\end{proof}

\section{Acknowledgements} 
\label{sec:acknowledgements}

Many thanks to Doug West for running the combinatorics REGS program and assisting with this project, as well as to Robert Jamison for introducing the list ranking problem and providing initial guidance.

The author acknowledges support from National Science Foundation grant DMS 0838434 EMSW21MCTP: Research Experience for Graduate Students.


\begin{thebibliography}{00}

\bibitem{BDJKKMT}
Bodlaender, H. L.; Deogun, J. S.; Jansen, K.; Kloks, T.; Kratsch, D.; Muller, H.; Tuza, Z. 
\newblock Rankings of graphs, Siam J. Discrete Math, Vol. 11, No. 1 (1998), 168--181.

\bibitem{BH}
Bruoth, E.;  Hor\u{n}\'{a}k, M.
\newblock Online-ranking numbers for cycles and paths, Discuss. Math. Graph Theory, 19 (1999), 175--197.

\bibitem{D}
Dereniowski, D.
\newblock The complexity of list ranking of trees, Ars Combin. 86 (2008), 97--114 .

\bibitem{IRV}
Iyer, A. V.; Ratliff, H. D.; Vijayan, G.
\newblock Optimal node ranking of trees, Inform. Process. Lett. 28 (1988), 225--229.

\bibitem{IRV1}
Iyer, A. V.; Ratliff, H. D.; Vijayan, G.
\newblock Parallel assembly of modular products, Tech. Rept. 88--06, Production and Distribution Research Center, Georgia Institute of Technology, Atlanta, GA (1988).

\bibitem{IRV2}
Iyer, A. V.; Ratliff, H. D.; Vijayan, G.
\newblock On an edge ranking problem of trees and graphs, Discrete Applied Math, Vol. 30, Issue 1 (1991), 43–-52.

\bibitem{J}
Jamison, R. E.
\newblock Coloring parameters associated with rankings of graphs, Congr. Numer. 164 (2003), 111--127. 

\bibitem{M}
McDonald, D.
\newblock On-line vertex ranking of trees, submitted.

\bibitem{NO} 
Ne\u{s}et\u{r}il, J.; Ossona de Mendez, P.
\newblock Tree-depth, subgraph coloring and homomorphism bounds, European J. Combin. 27 (2006), No. 6, 1022–-1041.

\bibitem{S}
Schauz, U.
\newblock Mr. Paint and Mrs. Correct, Electron. J. Combin., Vol. 16, No. 1 (2009), R77.

\end{thebibliography}
\end{document}